\documentclass[12pt,a4paper]{amsart}
\usepackage{dsfont}
\usepackage{graphicx}
\usepackage{amssymb}
\usepackage[all]{xy}
\usepackage{amsmath}
\usepackage{tikz}
\usepackage[english]{babel}
\usepackage{hyperref}
\usepackage{thmtools}
\usepackage{thm-restate}

\usepackage{cleveref}

\usepackage{amssymb,amsmath,amsfonts,amsthm,amscd}
\usepackage{indentfirst,graphicx}
\usepackage[font={scriptsize},captionskip=5pt]{subfig}
\usetikzlibrary{matrix}
\usepackage[T1]{fontenc}
\usepackage[utf8]{inputenc}
\usepackage{geometry}
\usepackage{cite}
\usepackage{mathabx} 
\usepackage{color, xcolor}
\usepackage{graphics,graphicx,epsf}
\usepackage{float} 

\usetikzlibrary{patterns}
\usetikzlibrary{patterns.meta}
\tikzdeclarepattern{
	name=hatch,
	parameters={\hatchsize,\hatchangle,\hatchlinewidth},
	bounding box={(-.1pt,-.1pt) and (\hatchsize+.1pt,\hatchsize+.1pt)},
	tile size={(\hatchsize,\hatchsize)},
	tile transformation={rotate=\hatchangle},
	defaults={
		hatch size/.store in=\hatchsize,hatch size=5pt,
		hatch angle/.store in=\hatchangle,hatch angle=0,
		hatch linewidth/.store in=\hatchlinewidth,hatch linewidth=.4pt,
	},
	code={
		\draw[line width=\hatchlinewidth] (0,0) -- (\hatchsize,\hatchsize);
	}
}
\usepackage{faktor}

\newtheorem{theorem}{Theorem}[section]
\newtheorem*{thm}{Theorem}
\newtheorem{corollary}[theorem]{Corollary}
\newtheorem{proposition}[theorem]{Proposition}
\newtheorem{lemma}[theorem]{Lemma}

\newtheorem{definition}[theorem]{Definition}

\newtheorem{example}[theorem]{Example}
\newtheorem{remark}[theorem]{Remark}

\newcommand{\C}{\ensuremath{\mathbb{C}}}
\newcommand{\R}{\ensuremath{\mathbb{R}}}

\newcommand{\Z}{\ensuremath{\mathbb{Z}}}

\newcommand{\Poli}{\ensuremath{\Gamma_{+}}}
\newcommand{\On}{\ensuremath{\mathcal{O}_{n}}}
\newcommand{\OX}[1]{\ensuremath{\mathcal{O}_{X(#1)}}}
\newcommand{\coned}[1]{\ensuremath{c\widecheck{on}e(#1)}}

\begin{document}

	\title[Newton polyhedra and the integral closure of ideals]{Newton polyhedra and the integral closure of ideals on affine toric varieties}
	\author{Amanda S. Araújo }
	\author{Thaís M. Dalbelo}
	\author{Thiago da Silva}

	\begin{abstract}
		In this work, we extend some results from polynomial rings to the broader context of semigroup rings. More precisely, we generalize Saia's characterization of Newton non-degenerate ideals to the context of ideals in $\mathcal{O}_{X(S)}$, where $X(S)$ is an affine toric variety defined by the semigroup $S\subset \mathbb{Z}^{n}_{+}$. We explore the relationship between the integral closure of ideals and the Newton polyhedron. We introduce and characterize non-degenerate ideals, showing that their integral closure is generated by specific monomials related to the Newton polyhedron. 
	\end{abstract}
	
	\maketitle
	\tableofcontents
	
	\section*{Introduction}
	
	The search for elements and tools to describe equisingularity conditions in families of analytic varieties is one of the main questions in Singularity Theory. The theory of the integral closure of ideals and modules provides a handy tool for studying equisingularity problems. For instance, in \cite{Tessier1} Teissier used the integral closure of the ideals to study the equisingularity
	of hypersurface germ families. Inspired by Teissier's work Gaffney used the integral closure to study the equisingularity of families of complete intersection with isolated singularities (ICIS) \cite{Gaffney1992}.
	
	Although essential, it is not easy to compute the integral closure of an ideal. However, there are some cases where this computation is possible. For example, if $I$ is an ideal generated by monomials on  $\mathcal{O}_n$ (the ring of germs of analytic functions $f: (\mathbb{C}^n,0) \to (\mathbb{C},0)$), the integral closure $\overline{I}$ of $I$ is generated by the monomials $x^k = x^{k_1} \dots x^{k_n}$, $k$ belonging to the Newton polyhedron of $I$. This result was extended by Saia \cite{Marcelo}, thus characterizing the class of Newton non-degenerate ideals $I \subset \mathcal{O}_n$. In \cite{Carles2004} Bivià-Ausina computes the integral closure of Newton non-degenerate submodules $M \subset \mathcal{O}_n^p$, $p \geq 1$.
	Therefore, the author characterize a class of submodules  $M \subset \mathcal{O}_n^p$ such that the integral
	closure $\overline{M}$ of $M$ is easily computable.
	
	In this work we generalize the results presented by Saia in \cite{Marcelo} to the case of non-degenerate ideals in $\mathcal{O}_{X(S)}$, in which $X(S)$ is the affine toric variety defined by the semigroup $S \subset \mathbb{Z}^n_{+}$. 
	
	This work is organized as follows. In Section 1, we introduce some definitions and properties related to affine toric varieties and Newton polyhedron defined in the $cone(S)$. In Section 2, we present a relationship between the integral closure of ideals in $\mathcal{O}_{X(S)}$ and the Newton polyhedron of $I$, proving the following theorem:
	
	\begin{thm}{(Theorem \ref{22})}
		Let $I\subset \mathcal{O}_{X(S)}$ be a monomial ideal. If $h \in \overline{I}$ then $supp(h) \subset \Gamma_+(I)$.
	\end{thm}
	
	In Section 3, we introduce the class of non-degenerate ideals of  $\mathcal{O}_{X(S)}$  and establish a relationship between the Newton polyhedron of the ideal $I$ and $C(\overline{I})$, which is the Newton polyhedron of the ideal generated by all monomials that belong to the integral closure of the ideal $I$.
	
	In Section 4, we characterize these ideals by showing that their 
	integral closure is generated by specific monomials related to the Newton polyhedron, thereby proving our main theorem: 
	\begin{thm}{(Theorem \ref{theo48})}
		Let $I$ be an ideal of $\mathcal{O}_{X(S)}$. Then $I$ is non-degenerate if and only if $\Gamma_{+}(I)=C(\overline{I}).$
	\end{thm}
	
	Ultimately, we use  Theorem \ref{theo48} to present an example of a family which satisfy the Whitney conditions (see Example \ref{ExW}). More precisely, we determine that a stratification, different from the one given by the torus action orbits, is a Whitney stratification.
	
	\section{Preliminary Notions and Results}

	For the convenience of the reader and to fix some notation we review some 
	general facts  in order
	to establish our results.
	
	\subsection{Toric varieties.}
	
	We introduce some basic concepts about toric varieties. These concepts can be found in  \cite{bra, Cox}.

	Let $S \subset \mathbb{Z}^n$ be a semigroup generated by a finite set $S_0 = \{b_1,\dots,b_r\} \subset \mathbb{Z}^n$ satisfying $\mathbb{Z}S_0 = \mathbb{Z}^n$.
	
	The set $S_0$ induces a group homomorphism $\pi_{S_0} : \mathbb{Z}^r \to \mathbb{Z}^n$ given by
	$$
	\pi_{S_0}((\alpha_1,\dots,\alpha_r)) = \alpha_1 b_1 + \dots + \alpha_r b_r.
	$$
	
	Given $\alpha = (\alpha_1, \dots, \alpha_r) \in \ker(\pi_{S_0})$ we set 
	$$
	{\alpha}_+ = \displaystyle \sum_{\alpha_i > 0} \alpha_i e_i \ \ \ \ \text{and} \ \ \ \ {\alpha}_{-} = -\displaystyle \sum_{\alpha_i < 0} \alpha_i e_i,
	$$
	where $e_1, \dots, e_r$ are the elements of the standard basis of $\mathbb{R}^r$. Let us observe that $\alpha = \alpha_{+} - \alpha_{-}$ and that $\alpha_{+}, \alpha_{-} \in \mathbb{N}^r$.
	
	Consider the ideal 
	\begin{equation}\label{nucleo}
		I_S = \langle x^{\alpha_{+}} - x^{\alpha_{-}}; \ \ \alpha \in \ker(\pi_{S_0}) \rangle \subset \mathbb{C}[x_1,\dots,x_r],
	\end{equation}
	where $x^{\beta} = x_1^{\beta_1} \dots x_r^{\beta_r}$ with $\beta = (\beta_1, \dots, \beta_r) \in \mathbb{N}^r$.

	\begin{definition}
		The $n$-dimensional {affine toric variety} $X(S)$ is defined by the 
		zero set of $I_S$, i.e.
		$X(S)=V(I_S) \subset \mathbb{C}^r$.
	\end{definition} 
	
	The ideal $I_S$ is a prime ideal (see \cite[Proposition $1.1.9$]{Cox}) and  $X(S) $ is the affine variety defined by $I_S$, which is not necessarily
	normal.

	There is a relationship between the toric variety $X(S)$ associated with $S$ and the algebraic $n$-dimensional torus $\mathbb{T} = (\C^{*})^{n}$. More precisely, the action of the torus $\mathbb{T}$ on $X(S)$ is given by
	\begin{equation}\label{action}
		\begin{matrix}\phi: &  (\C^{*})^n \times X(S) &\longrightarrow& X(S)\\ 
			&(t,x)&\longmapsto& (t^{b_1} x_1,\dots, t^{b_r} x_r)\end{matrix}
	\end{equation}
	where $t=(t_{1},\dots,t_{n})$, $b_{i}=(b^{1}_{i},\dots,b^{n}_{i})$ and $x=(x_{1},\dots,x_{r})$.

	Moreover, $\phi$ has an open and dense orbit $\mathcal{O}$ in $X(S)$ that is diffeomorphic to $(\C^*)^{n}$. Thus the $n$-dimensional toric variety $X(S)$ contains the torus $\mathbb{T}$ as a Zariski open dense subset.\\
	
	In the following, we introduce the concepts of cone and dual cone, which are studied in convex geometry. Those elements will be necessary to the definition of Newton polyhedra which we will use in this work.
	\begin{definition}
		Let $\{b_{1},\dots,b_{r}\}\subset \Z^{n}_{+}$ be a finite generator set of $S$. The convex polyhedral cone associated with $S$ in $\R^{n}$ is the set 
		$$ cone(S)=\Bigg\{ \sum_{i=1}^{r} \lambda_{i}b_{i} \ \vert \ \lambda_{i} \in \R \ , \ \lambda_{i}\geq 0 \Bigg\}.$$
		The vectors $b_{1},\dots,b_{r}$ are called generators of the  $cone(S)$. Also, set $cone(\emptyset)=\{0\}$.
	\end{definition}
	
	In what follows, we will consider that  $dim( cone(S))=n$ and that $cone(S)$ is strongly convex, i.e. $cone(S) \cap (- cone(S)) = \{0\}$.
	
	Let $(\R^{n})^{\ast}$ be the dual space of $\R^{n}$. To each cone, we associate the dual cone, defined by $$\coned{S}=\{ u \in (\R^{n})^{*} : \langle u,v \rangle \geq 0 , \forall v \in cone(S)\}.$$

	\subsection{Newton polyhedra.} In this section we present the definition of Newton polyhedron of ideals of the ring $\mathcal{O}_{X(S)}$. For the definition of Newton polyhedron of a function $f: X(S) \to \mathbb{C}$ we can refer to \cite{MT1} for instance.

	\begin{definition}
		The Newton polyhedron determined by $A \subset S$, denoted by $\Poli(A)$,  is the convex hull of the set  $\{k+v \ : \ k \in A , v \in cone(S)\}$. If $v\in \coned{S}$, we define $$\ell(v,\Poli(A))=\min\{ \langle  k,v\rangle \ : \ k \in \Poli(A) \}$$ and $$\Delta(v, \Poli(A))= \{k \in \Poli(A): \langle k,v\rangle=\ell(v,\Poli(A))\}.$$
		
		A subset $\Delta \subseteq \Poli(A)$ is called a face of $\Poli(A)$ when there exists  $v \in \coned{S}$ such that $\Delta= \Delta(v,\Poli(A))$.
	\end{definition}
	For the sake of simplicity, when it is clear from the context that the polyhedron $\Poli(A)$ is fixed, we will omit it from the notation and write $\Delta(v)$ and $\ell(v)$ instead of $\Delta(v,\Poli(A))$ and $\ell(v,\Poli(A))$.
	
	\begin{figure}[H]
		
		\begin{tikzpicture}[scale=0.6]
			
			\draw[->] (-0.5,0) -- (4,0);
			\draw[->] (0,-0.5) -- (0,4);

			\draw[lightgray, thick] (0,0) -- (2,4);
			
			\fill[lightgray, opacity=0.2] (0,0) -- (2,4) -- (4,4) --(4,0)  --cycle;
			
			\fill[black] (1.7,2.5) circle (0.6mm) ;
			\fill[black] (1.5,1) circle (0.6mm) ;
			\fill[black] (3,0.5) circle (0.6mm) ;
			
		\end{tikzpicture}
		\begin{tikzpicture}[scale=0.6]
			
			\draw[->] (-0.5,0) -- (4,0);
			\draw[->] (0,-0.5) -- (0,4);

			\draw[lightgray, thick] (0,0) -- (2,4);
			
			\fill[lightgray, opacity=0.2] (0,0) -- (2,4) -- (4,4) --(4,0)  --cycle;
			
			\fill[black] (1.7,2.5) circle (0.6mm) ;
			\fill[black] (1.5,1) circle (0.6mm) ;
			\fill[black] (3,0.5) circle (0.6mm) ;

			\fill[purple, opacity=0.2] (1.5,1) -- (3,4) -- (4,4) --(4,1) --cycle;
			\fill[purple, opacity=0.2] (1.7,2.5) -- (2.5,4) -- (4,4) --(4,2.5) --cycle;
			\fill[purple, opacity=0.2] (3,0.5) -- (4,2.5) -- (4,0.5) --cycle;
			
		\end{tikzpicture}
		\begin{tikzpicture}[scale=0.6]
			
			\draw[->] (-0.5,0) -- (4,0);
			\draw[->] (0,-0.5) -- (0,4);

			\draw[lightgray, thick] (0,0) -- (2,4);
			
			\fill[lightgray, opacity=0.2] (0,0) -- (2,4) -- (4,4) --(4,0)  --cycle;
			
			\fill[black] (1.7,2.5) circle (0.6mm) ;
			\fill[black] (1.5,1) circle (0.6mm) ;
			\fill[black] (3,0.5) circle (0.6mm) ;

			\fill[purple, opacity=0.2] (1.5,1) -- (3,4) -- (4,4) --(4,1) --cycle;
			\fill[purple, opacity=0.2] (1.7,2.5) -- (2.5,4) -- (4,4) --(4,2.5) --cycle;
			\fill[purple, opacity=0.2] (3,0.5) -- (4,2.5) -- (4,0.5) --cycle;

			\draw[purple, thick]  (1.7,2.5) -- (2.5,4);
			\draw[purple, thick] (1.7,2.5) -- (1.5,1);
			\draw[purple, thick] (3,0.5) -- (1.5,1);
			\draw[purple, thick] (4,0.5) -- (3,0.5);

			\pattern[pattern={hatch[hatch size=10pt, hatch linewidth=.2pt, hatch angle=90]} , pattern color=purple]  (1.7,2.5) -- (2.5,4)--(4,4)--(4,0.5)--(3,0.5)--(1.5,1)--cycle;
			
		\end{tikzpicture}
		\caption{ Newton polyhedron.}
		
	\end{figure}
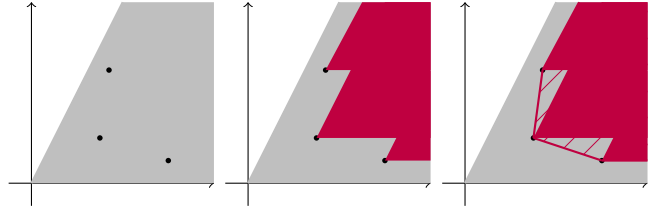
	Since a polyhedron can be described as the intersection of the half-planes defined by its faces, we have an important characterization.
	\begin{lemma}
		Let $\Poli(A) \subseteq cone(S)$ be the Newton polyhedron of a set $A \subset S$. Then 
		$$ \Gamma_{+}(A) = \{ x \in cone(S):  \langle x,v \rangle \geq \ell(v, \Gamma_{+}(A)), \mbox{ for all } v \in \coned{S}\}.$$
	\end{lemma}
	
	In the sequence, we define the Newton polyhedron of an ideal $I$ of the ring $\OX{S}$. First, let us introduce the Newton polyhedron of germs in $ \OX{S}$.
	
	\begin{definition}
		Let $g \in \OX{S}$, i.e., $g=f\vert_{X(S)}$, where $f: \C^{r} \to \C$ is an analytic function such that $f=\sum_{k} a_{k}x^{k}$ is its Taylor expansion. The support of $g \in \OX{S}$ is defined by the set $$supp (g) = \big\{k_1b_1+\cdots+k_rb_r \in S: a_k \neq 0 \text{ and }k=(k_1,\dots,k_r)\big\}, $$
		where $b_1,\dots, b_r$ are generators of $S$. The Newton polyhedron of $g$ is defined by $\Poli(g)= \Poli(supp(g))$.
	\end{definition}
	If $v \in \coned{S}$, we define  $\ell(v,g)=\ell(v,\Poli(g))$ and $\Delta(v,g)=\Delta(v,\Poli(g))$.
	
	Consider $G \subseteq \OX{S}$. We denote by $supp(G)$ the union of the supports of the elements of $G$ and define the Newton polyhedron of $G$ as $\Poli(G)=\Poli(supp(G))$. If $I \subseteq \OX{S}$ is an ideal generated by $G$, then $\Poli(I)=\Poli(G)$. That is, if $G = \{g_{1},\dots,g_{p}\}$ is a generating system of $I$, then $\Poli(I)$ is equal to the convex hull of $\Poli(g_{1}) \cup \cdots\cup \Poli(g_{p})$.\\
	
	As in \cite{Carles2004}, we denote by $I^{\circ}$ the ideal of $\OX{S}$ generated by the monomials $x^{k}$ such that
	$k_{1}b_{1}+\cdots+k_{r}b_{r} \in \Poli(I)$, where $x_1,\dots,x_r$ is a coordinate system in $X(S)$.

	\section{The integral closure of ideals in \texorpdfstring{ $\mathcal{O}_{X(S)}$}{OXS}}
	Let $I\subseteq \OX{S}$ be an ideal. An element $h \in \OX{S}$ is said to be integral over $I$ if it satisfies an integral dependence relation 
	$$h^{m}+a_{1}h^{m-1}+ \cdots + a_m=0,\text{ with } a_{i}\in I^{i}.$$
	
	The set of elements $h \in \OX{S}$ which are integral over $I$ forms an ideal of $\OX{S}$, called the integral closure of $I$. We denote the integral closure of $I$ by $\overline{I}$. If $\overline{I} = I$, then $I$ is called integrally closed.

		Let $X$  be a complex analytic set and $I$ an ideal in $\mathcal{O}_{X,x}$. In \cite{Tessier1} the following equivalences are proved.
	
	\begin{proposition}\textup{\cite{Gaffney1992}} \label{prop21}Let $I$ be an ideal of $\mathcal{O}_{X,x}$. The following statements are equivalent:\
		\begin{itemize}
			\item[(i)] $h \in \overline{I}$.
			\item[(ii)] (Growth condition) For each choice of generators $\{g_{i}\}$ of $I$ there exists a neighborhood $U$ of $x$ in $X$ and a constant $\mathcal{C}>0$ such that  $$\vert h(y)\vert \leq  \mathcal{C} \cdot \sup_{i} \{\vert g_{i}(y) \vert\}, \ \forall y \in U.$$
			\item[(iii)](Valuative criterion) For each analytic curve $\psi: (\C,0) \to (X,x)$, 
			$h\circ \psi$ lies in $\psi^{\ast}(I)\mathcal{O}_{1}$, where $\psi^{\ast}(I)\mathcal{O}_{1}$ is the ideal generated by $\psi^{\ast}(I)=\{i \circ \psi : i \in I \}$ in the local ring $\mathcal{O}_1$.
		\end{itemize}
		
	\end{proposition}

	Based in \cite{Tessier1} and considering monomial ideals in $\OX{S}$, we are going to prove that if $h \in \overline{I}$, then all monomials appearing  
	in $h$ are represented in $cone(S)$ by points located in $\Poli(I)$.
	
	\begin{theorem}
		Let $I\subset \mathcal{O}_{X(S)}$ be a monomial ideal. If $h \in \overline{I}$ then $supp(h) \subset \Gamma_+(I)$.\label{22}
	\end{theorem}
	\begin{proof}
		Consider $h \in \overline{I} \subseteq \mathcal{O}_{X(S)}$. Note that $h=f\vert _{X(S)}$, where $f: \C^{r} \to \C$ is an analytic function with Taylor expansion $f= \sum_{k}a_kx^k $, $k\in \Z^{r}_+$.
		Also observe that $I$ is a monomial ideal in the semigroup ring generated by $S \subseteq \mathbb{Z}_+^r$. According to  \cite[Theorem $4.45$]{BrunsGubeladze2009}, $ \overline{I} $ is also a monomial ideal. Consequently, $x^k \in \overline{I}$, for all $x^k$ that appears in $f= \sum_{k}a_kx^k $.

		Let us denote by $d=(d_1,\dots,d_n)$ the element of the form $d=\tilde{k}_1b_1+\cdots+\tilde{k}_rb_r \in supp(h)$     for some $\tilde{k}=(\tilde{k}_1,\dots,\tilde{k}_r) \in \Z^{r}_+$ such that $a_{\tilde{k}}\neq 0$.
		Let $\Delta$ be a face of $\Gamma_+(I)$. Then, there exist $\gamma=(\gamma_1,\dots, \gamma_n) \in \coned{S}$ and $\eta \geq 0$ such that
		$$\Delta=\{c=(c_1,\dots,c_n) \in \Gamma_+(I): \langle c, \gamma\rangle=\eta \}.$$
		
		Let $\varphi: \C \to X(S)$ be an analytic curve defined by $$\varphi(t)=( t^{\langle b_1, \gamma \rangle},\dots, t^{\langle b_r , \gamma \rangle}).$$
		
		Since $X(S)=V(I_S)$ is a toric variety, then $I_S$ is a binomial ideal (see \cite{Cox}). Therefore, $(1,\dots,1)\in X(S)$. Note that, for $t\neq 0$, we have  $(t^{\gamma_1},\dots,t^{\gamma_n}) \in (\mathbb{C}^*)^n$. Considering the toric action defined in (\ref{action}), we obtain $$\varphi(t)= \phi((t^{\gamma_1},\dots,t^{\gamma_n}) ,(1, \dots,1))=( t^{\langle b_1, \gamma \rangle},\dots, t^{\langle b_r, \gamma \rangle}) \in X(S).$$
		
		So, we have
		$$x^{\tilde{k}} \circ \varphi(t)=t^{\langle b_1,\gamma\rangle \tilde{k}_1} \cdots t^{\langle b_r,\gamma\rangle \tilde{k}_r } = t^{N},$$
		where $N=\langle b_1, \gamma \rangle\tilde{k}_1+\cdots +\langle b_r, \gamma \rangle \tilde{k}_r$.
		
		Let $\upsilon: \mathcal{O}_1 \to \Z \cup \{\infty\}$ be the usual valuation on $\mathcal{O}_1$. Thus, we obtain, $$\upsilon(x^{\tilde{k}} \circ \varphi)= \langle b_1, \gamma \rangle\tilde{k}_1+\cdots +\langle b_r, \gamma \rangle \tilde {k}_r = N.$$
		Since $x^{\tilde{k}} \in \overline{I}$, then $x^{\tilde{k}} \circ \varphi \in \varphi^*(I)\mathcal{O}_1$, i.e., $h \circ \varphi = f_1(q_1 \circ \varphi)+\dots+ f_l(q_l \circ \varphi)$, where $q_1,\dots, q_l \in I$ and $f_1,\dotsc,f_l \in \mathcal{O}_1$. We obtain
		\begin{eqnarray}\label{valuation2}
			\upsilon( x^{\tilde{k}} \circ \varphi) \geq \min_{i}\{ \upsilon(f_i(q_i \circ \varphi))\}&=& \upsilon(f_j(q_j \circ \varphi))\nonumber\\
			&\geq& \upsilon(f_j)+ \upsilon(q_j \circ \varphi)\nonumber\\
			&\geq& \upsilon(q_j \circ \psi)\nonumber\\
			&\geq& \upsilon(g \circ \psi),
		\end{eqnarray}
		where $g$ is a generator of $I$. Taking $g=x^w$ such that $w_1b_1+\cdots +w_rb_r \in \Delta$, we have
		$x^w \circ \varphi(t) =t^{\langle b_1, \gamma \rangle w_1+\cdots +\langle b_r, \gamma \rangle w_r }.$ Then $$\upsilon(x^w\circ \varphi) = w_1\langle b_1, \gamma \rangle +\cdots +w_r\langle b_r, \gamma \rangle =\langle w_1b_1+\cdots +w_rb_r, \gamma\rangle=\eta .$$
		
		Moreover, based on the inequality (\ref{valuation2}), we have
		$$\langle d, \gamma \rangle =N = \upsilon(x^{\tilde{k}}\circ \varphi) \geq \upsilon(x^w \circ \varphi)=\eta.$$
		
		Then $d$ is above $\Delta.$ Therefore, $supp(h) \subset \Gamma_+(I)$.
	\end{proof}
	
	Using this result, we can prove the following corollary.
	
	\begin{corollary}\label{corollary23}
		Let $I$ be an ideal of $\OX{S}$. Then we have:
		\begin{itemize}
			\item[(i)]$I^{\circ}$ is integrally closed.
			\item[(ii)] $\Poli(I)=\Poli(\overline{I})$
		\end{itemize}
		\begin{proof}
			$(i)$ Let $h \in \overline{I^{\circ}}$. By Proposition \ref{22}, $supp(h)\subset \Gamma_+(I^{\circ})$. As $\Gamma_+(I^{\circ})=\Gamma_+(I)$, the supports of the monomials that appear in $h= \sum_k a_k x^k$ are in $\Gamma_+(I)$, i.e., $k_1b_1+\cdots+k_rbr \in \Gamma_+(I) $. By definition, these monomials are in $I^{\circ}$. It follows that $h \in I^{\circ}$.
			
			$(ii)$ By definition, $I \subset I^{\circ}$. So $\overline{I}\subset \overline{ I^{\circ}}$. It follows from $(i)$ that $\overline{I} \subset \overline{I^{\circ}} = I^{\circ}$. Thus, we have $I \subset \overline{I} \subset I^{\circ}$ and, then, $\Gamma_+(I) \subset \Gamma_+(\overline{I}) \subset \Gamma_+(I^{\circ} )=\Gamma_+(I).$    \end{proof}
	\end{corollary}
	
	In \cite[Prop. 3.8]{bisui2024rationalpowersinvariantideals} the authors also study the integral closure of monomial ideals in $\mathcal{O}_{X(S)}$. We remark that in our work the monomial hypothesis is not necessary to obtain the related results.
	
	According to item $(ii)$ of Corollary \ref{corollary23}, we can affirm that $\overline{I} \subset I^{\circ}$. However, it is not always true that $I^{\circ} \subset \overline{I}$, as we can see in the following example.
	\begin{example}\label{24}
		Let $S$ be the semigroup generated by the set $\{(1,0),(1,1),(1,2) \} \subset \Z^{2}_{+}$. Then, one has that $X(S) = V(I_S) \subset {\mathbb{C}^3}$ is a toric surface, where $I_S$ is the ideal generated by the binomial $xz - y^2$. Consider  $I \subset \mathcal{O}_{X(S)}$ the ideal generated by $f=x^2y+3x^2z $ and $ g =y^3-xy+z^2$. Then, $supp(I)=\{(3,1),(3,2),(3,3),(2,1),(2,4)\},$ and the Newton polyhedron of $I$ is represented in the Figure \ref{Fig2}. Note that $y^2 \in I^{\circ}$, since  $(2,2) \in \Poli(I)$. But $y^2 \notin \overline{I}$, because there is an analytic  curve $\psi: (\C,0) \to (X(S),0)$ defined by $\psi(t)=(t,t,t)$ such that $y^2 \circ \psi\notin\psi^*(I)\mathcal{O}_1$.
		\begin{figure}[H]
			\centering
			\begin{tikzpicture}[scale=0.6]
				\draw[->] (-0.5,0) -- (5,0);
				\draw[->] (0,-0.5) -- (0,5);
				
				\draw[lightgray, thick] (0,0) -- (2.5,5);
				
				\fill[lightgray, opacity=0.2](0,0) -- (2.5,5) -- (5,5) --(5,0)  --cycle;
				
				\fill[black] (0,2) node[left] {$2$};
				\fill[black] (0,3) node[left] {$3$};
				\fill[black] (0,4) node[left] {$4$};
				\fill[black] (0,1) node[left] {$1$};
				\fill[black] (1,0)  node[below] {$1$};
				\fill[black] (2,0)  node[below] {$2$};
				\fill[black] (3,0)  node[below] {$3$};
				\fill[black] (4,0)  node[below] {$4$};

				\fill[black] (2,4) circle (0.4mm) ;
				\fill[black] (3,2) circle (0.4mm) ;
				\fill[black] (3,1) circle (0.4mm) ;
				\fill[black] (3,3) circle (0.4mm) ;
				\fill[black] (2,1) circle (0.4mm) ;
				
				\draw[darkgray, thick] (2,4) -- (2,1);
				\draw[darkgray, thick] (2,4) -- (2.5,5);
				\draw[darkgray, thick] (2,1) -- (5,1);
				\fill[gray, opacity=0.2] (2,4) -- (2.5,5) -- (5,5) -- (5,4)--cycle;
				\fill[gray, opacity=0.2] (3,3) -- (4,5) -- (5,5) -- (5,3)--cycle;
				\fill[gray, opacity=0.2] (3,2) -- (4.5,5) -- (5,5) -- (5,2)--cycle;
				\fill[gray, opacity=0.2] (3,1)  -- (5,5) -- (5,1)--cycle;
				\fill[gray, opacity=0.2] (2,1) -- (4,5) -- (5,5) --(5,1)--cycle;
				\pattern[pattern={hatch[hatch size=7pt, hatch linewidth=.2pt, hatch angle=90]} , pattern color=darkgray] (2,4) -- (2.5,5) --(5,5)--(5,1)--(2,1) --cycle;
				\fill[red] (2,2) circle (0.7mm) ;
			\end{tikzpicture}
			\caption{Newton polyhedron of $I=\langle x^2y+3x^2z,y^3-xy+z^2\rangle$.}
			\label{Fig2}
		\end{figure}
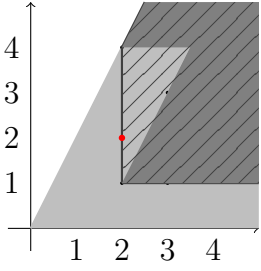
	\end{example}

	\section{Non-degenerate Ideals}
	
	In this section, we present the concept of non-degenerate ideals and establish a relationship between the Newton polyhedron of the ideal $K_{I}$, which is generated by all monomials that belong to the integral closure of the ideal $I$, and the Newton polyhedron of $I$.
	
	Let $g \in \OX{S}$. As we said before, this means that $g=f\vert_{X(S)}$, where $f: \C^{r} \to \C$ is an analytic function, and then can be written as $f=\sum_{k} a_{k}x^{k}$. Now, given a compact subset $A \subseteq cone(S)$, we denote by
	$$g_{A}:= \sum_{k_{1}b_{1}+\cdots+k_{r}b_{r} \in A\cap supp(g)} a_{k} x^{k}.$$
	We set $g_{A} = 0$ whenever $A \cap supp(g) = \emptyset$.
	
	Let $g$ be a representative of the germ $g \in \mathcal{O}_{X(S)}$ such that $0 \notin supp(g)$. Then, we associate to $g$ a polynomial $L(g) \in \On$ given by
	$$L(g)= \sum_{v \in supp(g)} a_v z^v,$$
	where $a_v=a_k$, $a_k\neq 0$ and $v=k_1b_1+\cdots+ k_rb_r.$

	\begin{definition}
		Let $G=\{g_{1},\dots,g_{l}\} \subseteq \OX{S}$. We say that $G$ is non-degenerate if, for each compact face $\Delta$ of $\Poli(G)$, the equations $$L\big((g_1)_{\Delta}\big)= \cdots= L\big((g_l)_{\Delta}\big)=0$$
		have no common solution in $(\C^{\ast})^{n}$.
		
		An ideal $I\subseteq \OX{S}$ is non-degenerate if $I$ admits a non-degenerate system of generators.
		
	\end{definition}
	\begin{remark}
		Upon considering $S=\langle e_{1},\dots,e_{n} \rangle$, the toric variety $X(S)$ corresponds to $\C^{n}$. In this case, we have $cone(S)=\R^{n}_{+}$. Thus, the previous definition can be seen as a generalization of the definition of Newton non-degenerate in $\mathcal{O}_{n}$.
	\end{remark}
	
	In the definition above we ask only for algebraic conditions. This will be enough for our purposes. However, in the literature, we can find other definitions that also require geometrical conditions. For instance, in \cite{MT1} Matsui and Takeuchi define the concept of Newton non-degenerate functions and complete intersection to study Milnor fiber over singular toric varieties. For this purpose, they also used geometrical conditions. 
	
	Yoshinaga proved in \cite{Yoshinaga} that a function $f\in \On$ is Newton non-degenerate if and only if the integral closure of the ideal $$I(f)=\Big\langle x_{1}\frac{\partial f}{\partial x_{1}}, \dots, x_{n}\frac{\partial f}{\partial x_{n}}\Big\rangle$$ is generated by the monomials $x^{k}$ such that $k \in \Poli(f)$. Saia extended this result in  \cite{Marcelo}, proving that an ideal $I \subseteq \On$ is Newton non-degenerate if and only if $\overline{I}$ is generated by the monomials $x^{k}$ such that $k \in \Poli(I).$
	In the following section, we extend this result to the ideals in $\OX{S}$. 
	
	\begin{definition}Let $I$ be an ideal in $\mathcal{O}_{X(S)}$. We define the set $C(\overline{I})$ as the convex hull in $cone(S) $ of the set $$\bigcup \left\{\lambda_1b_1+\cdots+\lambda_r br : x^\lambda \in \overline{I}\right\}, $$
		where $\lambda=(\lambda_1,\cdots,\lambda_r)$.
		\label{c(I)}
	\end{definition}
	
	Observe that $C(\overline{I})$ is the Newton polyhedron of the ideal $K_I$ generated by all monomials belonging to the integral closure of $I$, i.e., $K_I=\langle \{x^m: x^m \in \overline{I}\} \rangle$. This is because $\overline{I}$ is an ideal of $\mathcal{O}_{X(S)}$.
	
	\begin{lemma} \label{lemma34}
		Let  $I=\langle g_{1},\dots,g_{l}\rangle \subseteq \mathcal{O}_{X(S)}$ be the ideal generated by the set $\{g_1, \dots,g_l\} \subset \mathcal{O}_{X(S)}$. Then $C(\overline{I})\subseteq \Poli(I)$.
	\end{lemma}
	\begin{proof}
		If $d=\lambda_1 b_1+\cdots+\lambda_rb_r \notin \Poli(I)$, there exists some $v \in \coned{S}$ such that $$\langle d,v\rangle < \ell(v,\Poli (I)).$$ 
		By the definition of $\ell(v,\Poli(I))$, we have
		$$\ell(v,\Poli(I)) \leq \langle\beta_{j1}b_1+ \cdots+ \beta_{jr}b_r,v\rangle,$$
		where $\beta_{j1}b_1+ \cdots+ \beta_{jr}b_r\in supp(g)\subset \Poli(I)$, for all $g \in I$.
		Then \begin{equation}
			\lambda_1\langle b_1,v\rangle+\cdots+ \lambda_r\langle b_r,v\rangle=\langle d,v \rangle < \beta_{j1}\langle b_1,v\rangle+\cdots+ \beta_{jr}\langle b_r,v\rangle. \label{5}
		\end{equation}
		
		Consider $\varphi: \C \to X(S)$ defined by $$\varphi(t)=(t^{\langle b_1,v\rangle}, \dots, t^{\langle b_r,v \rangle}).$$
		
		Since $X(S)=V(I_S)$ is a toric variety, $I_S$ is a binomial ideal. So, $(1,\dots,1)\in X(S)$. Note that $(t^{v_1},\dots,t^{v_n}) \in (\mathbb{C}^*)^n$, for $t\neq 0$. Thus, considering the toric action defined in (\ref{action}), we obtain $$\varphi(t)= \phi((t^{v_1},\dots,t^{v_n}) ,(1,\dots ,1))=( t^{\langle b_1, v\rangle},\dots, t^{\langle b_r, v\rangle}) \in X(S).$$
		
		Furthermore, $\varphi(0)=(0,\dots, 0) \in X(S) $, since $\dim(cone(S))= n$ and $cone(S)$ is strongly convex.
		We have
		$$x^\lambda \circ \varphi= t^{\lambda_1\langle b_1,v\rangle+\cdots+ \lambda_r\langle b_r,v\rangle},$$
		and for all $g \in I$, we can write
		$$ g = \sum_{\beta_{j1}b_1+ \cdots+ \beta_{jr}b_r\in supp(g)}a_j x^{\beta_j}.$$
		
		Thus,
		$$ g \circ \varphi= \sum_{\beta_{j1}b_1+ \cdots+ \beta_{jr}b_r\in supp(g)}a_j t^{\beta_{j1}\langle b_1,v\rangle+\cdots+ \beta_{jr}\langle b_r,v\rangle}.$$
		
		By inequality (\ref{5}), $x^\lambda \circ \varphi \notin \varphi^*(I)\mathcal{O}_1$. Therefore, $x^\lambda \notin \overline{I}$.
		Consequently
		$$d \notin\bigcup \left\{\lambda_1b_1+\cdots+\lambda_r br : x^\lambda \in \overline{I}\right\}. $$
		
		We proved that $\bigcup \left\{\lambda_1b_1+\cdots+\lambda_r br : x^\lambda \in \overline{I}\right\} \subset \Poli(I)$. Hence, $C(\overline{I})\subset \Poli(I)$.
	\end{proof}
	
	\section{Toroidal Embedding}
	
	Our main goal is to prove that an ideal $I \subseteq \OX{S}$ is non-degenerate if and only if $C(\overline{I})=\Poli(I)$. To do this, we will consider the construction of the toroidal embedding associated with the Newton polyhedron $\Poli(I)$.
	
	In \cite{Marcelo}, Saia generalized Yoshinaga's result \cite{Yoshinaga} using the construction of a toroidal embedding associated with the Newton polyhedron of a finite codimensional ideal $I$ in $\On$ to show that $I$ is Newton non-degenerate if and only if the Newton polyhedron of $I$ is the convex hull of the set of $n$-tuples $m=(m_1,\dots,m_n)$ such that $x^m \in \overline{I}$. The procedure for constructing the toroidal embeddings associated with a given Newton polyhedron is a local modification of Khovanskii's method of assigning a compact complex non-singular toroidal manifold to an integer-valued compact convex polyhedron in $\C^n$. Here, we apply this procedure to the Newton polyhedron in the $cone(S)$. This construction is the essential tool used to prove our main result. 
	
	For any ideal $I \subseteq \mathcal{O}_{X(S)}$ we describe a partition of $\coned{S}$ into convex cones with respect to $\Gamma_{+}(I)$.
	
	\begin{definition}
		The vector $a=(a_1,\dots,a_n) \in \coned{S}$ is called a primitive integer vector if $a$ is the vector with minimum length in $C(a)\cap (\Z^{n}-\{0\})$, where $C(a)$ is the half ray emanating from $0$ passing  thorough $a$.
	\end{definition}
	
	We define an equivalence relation on $\coned{S}$ by $$a \sim a' \text{ if only if }\Delta(a)=\Delta(a').$$

	Considering this equivalence, any equivalence class is naturally identified with a convex cone with its vertex at zero, specified by finitely many linear equations and strictly linear inequalities with rational coefficients.
	
	The closures of equivalence classes specify a partition $\Sigma_{0}$ of $\coned{S}$ into closed convex cones that have the properties:
	\begin{itemize}
		\item If $\sigma_1$ is a face of a cone $\sigma$ in $\Sigma_0$, then $\sigma_{1} \in \Sigma_{0}$.
		\item  For any cones $\sigma_1$ and $\sigma_2$ in $\Sigma_{0}$, $\sigma_{1} \cap \sigma_{2}$ is a face of both $\sigma_{1}$ and $\sigma_{2}$.
	\end{itemize}
	
	In this way, based on $\Sigma_{0}$ and using the algorithm described in the proof of Theorem $11$ from \cite{Kempf1973}, we construct a partition $\Sigma$ of $\coned{S}$ into finitely many closed convex cones with their vertices at zero such that: 
	\begin{enumerate}
		\item[1.] Any cone belonging to $\Sigma$ lies in one of the cones in $\Sigma_{0}$ and is specified by finitely many linear equations and linear inequalities with rational coefficients.
		\item[2.] If $\sigma_1$ is a face of a cone $\sigma$ in $\Sigma$, then $\sigma_{1} \in \Sigma$.
		\item[3.]  For any cones $\sigma_1$ and $\sigma_2$ in $\Sigma$, $\sigma_{1} \cap \sigma_{2}$ is a face of both $\sigma_{1}$ and $\sigma_{2}$.
		
		\item[4.] Any cone $q$-dimensional $\sigma$ in $\Sigma$ is simplicial and unimodular, i.e., there exists a set of primitive integer vectors $a^{1}(\sigma),\dots,a^{q}(\sigma)$ which are linearly independent over $\R$ and $n-q$ primitive integer vectors $a^{q+1}(\sigma),\dots,a^{n}(\sigma)$ such that $$\Z a^{1}(\sigma)+ \cdots + \Z a^{n}(\sigma)=\Z^{n}.$$
	\end{enumerate}
	Based on these properties, we can make the following observations:
	
	For a $1$-dimensional cone $\sigma$ in $\Sigma$, where $a^{i}(\sigma)$ represents the corresponding primitive integer of $\sigma$, the set $\Delta(a^{i}(\sigma))$ is a closed face of dimension $(n-1)$ of $\Poli(I)$.
	
	For each compact face $\Delta $ of $\Poli(I)$, there exists a cone $\sigma \in \Sigma$ of dimension $n$ and a subset $J$ of $\{1,\dots,n\}$ such that $\Delta$ can be expressed as the intersection  $\bigcap_{j \in J} \Delta(a^{n}(\sigma))$, where  $a^{1}(\sigma),\dots,a^{n}(\sigma)$ is the corresponding set of primitive integer vectors of $\sigma$. We shall denote this face $\Delta$ by $\Delta_J$. Notice that $\Delta_{J}=\Delta\big(\sum_{i\in J}a^{i}(\sigma)\big)$.
	
	For all $n$-dimensional cones $\sigma \in \Sigma$, $\bigcap^{n}_{i=1} \Delta(a^{i}(\sigma))$ is a one-point set.\\

	Let $\sigma$ be a $n$-dimensional cone in $\Sigma$ and $a^{1}(\sigma),\dots,a^{n}(\sigma)$ the corresponding set of primitive integer vectors of $\sigma$ that has been ordered. We associate to each such $\sigma$ a copy of $\C^{n}$ denoted by $\C^{n}(\sigma)$. Let us  denote by $\pi_{\sigma}: \C^{n}(\sigma)\to X(S)$ the mapping given by 
	$$\pi_{\sigma}(y_1,\dots,y_n)=\left(y_1^{\langle a^1(\sigma),b_1 \rangle}  \cdots  y_n^{\langle a^n(\sigma),b_1 \rangle}, \dots,y_1^{\langle a^1(\sigma),b_r \rangle}  \cdots  y_n^{\langle a^n(\sigma),b_r\rangle}\right),$$ 
	where $y_1,y_2,\dots,y_n$ are the coordinates in $\C^n(\sigma)$ and $a^{i}(\sigma)=(a^{i}_{1}(\sigma),\dots,a^{i}_{n}(\sigma))$.\\
	
	Note that $\pi_{\sigma}$ is the composition $\tilde{\phi} \circ \Omega_{\sigma}$ of the following maps 
	\begin{displaymath}
		\begin{matrix}\Omega_{\sigma}: &\C^{n}(\sigma)&\longrightarrow& (\C^{\ast})^{n}\\ 
			& (y_1,\dots,y_n)&\longmapsto& \left(y_1^{ a^{1}_{1}(\sigma)}\cdots  y_n^{ a^{n}_{1}(\sigma)}, \dots,y_1^{a^{1}_{n}(\sigma)}  \cdots y_n^{a^{n}_{n}(\sigma)}\right)\end{matrix}
	\end{displaymath}
	and 
	\begin{displaymath}
		\begin{matrix}\tilde{\phi}: &(\C^{\ast})^{n}&\longrightarrow& X(S)\\ 
			& z&\longmapsto& (z^{b_1},\dots,z^{b_{r}}).\end{matrix}
	\end{displaymath}
	where $b_1,\dots,b_r$ are the generators of $S$. 
	The map $\tilde{\phi}$ is a restriction of the torus action given in (\ref{action}) to $(\mathbb{C}^*)^n$ and it is a diffeomorphism from $(\C^\ast)^n$ to the open dense orbit $\mathcal{O}$ of $X(S)$. In this context, $\pi_\sigma$ maps the $(\C^\ast)^n$ obtained from $\C^n(\sigma)$ onto the $(\C^\ast)^n$ corresponding to $\mathcal{O}$. Consequently, for every point $w \in \mathcal{O}$, there exists $y_\sigma \in \C^n(\sigma)$ such that $w = \pi_\sigma(y_\sigma)$.\\

	Consider the disjoint union $ \displaystyle\mathcal{C}= \bigcup_{\sigma \in \Sigma} \C^{n}(\sigma)$. We define the following equivalence relation:
	given two points $y_{\sigma} \in \C^{n}(\sigma)$ and $y_{\tau} \in \C^{n}(\tau)$, $$y_{\sigma}  \sim y_{\tau} \text{ if and only if } \pi_\sigma(y_\sigma)=\pi_\tau(y_\tau).$$
	We will denote this set thus obtained by $X=X(\Gamma_{+}(I))= 
	\mathcal{C}/\sim$.\\

	Since $\Sigma$ satisfies properties $1-4$, by Theorems $6$, $7$ and $8$ of \cite{Kempf1973}, we have that $X$ is a non-singular $n$-dimensional algebraic complex manifold and $\pi: X \to X(S)$ defined by $\pi(y)=\pi_{\sigma}(y_\sigma)$ is a proper analytic mapping onto $X(S)$, where $y_{\sigma} \in \C^{n}(\sigma)$ is a representative of the equivalence class $y \in X$.\\
	
	Let $J$ be a subset of $\{1,\dots,n\}$. Define 
	$$E_{\sigma,J}=\{y_{\sigma} \in \C^n(\sigma): y_{\sigma,j}=0, \forall j \in J\}\text{ and }E_{\sigma,J}^*=\{y_{\sigma} \in E_{\sigma,J}: y_{\sigma,j}\neq 0, \forall j \notin J\},$$
	where $y_\sigma =(y_{\sigma,1}, \dots, y_{\sigma,n})$.
	
	Based on \cite[Lemma 3.1]{Fukui1985}, we have the following lemma.
	
	\begin{lemma}\label{lemma42}
		$\pi_{\sigma}(E_{\sigma,J})=\{0\}$ if only if $\Delta_J= \bigcap_{i \in J} \Delta(a^i(\sigma))$ is a compact face of  $\Gamma_{+}(I)$.
	\end{lemma}
	\begin{proof}
		By definition of $\pi_\sigma$, we have $$\pi_\sigma(E_{\sigma,J})= \{0\} \text{ if and only if }  \displaystyle  \big\langle \sum_{i \in J}a^{i}(\sigma),b_l\big\rangle>0, \text{ for all } l \in \{1,2,\dots,r\}.$$
		
		Note that  $ \big\langle \sum_{i \in J}a^{i}(\sigma),b_l\big\rangle = 0$ for some $l \in \{1,\dots,r\}$ if and  only if $\Delta_J$ is a face of  $\Poli(I)$ parallel to some face of $cone(S)$, i.e., $\Delta_J$ is not compact.
		
		Therefore, $ \big\langle \sum_{i \in J}a^{i}(\sigma),b_l\big\rangle>0, \text{ for all } l \in \{1,2,\dots,r\}$ if, and only if, $\Delta_J$ is a compact face.
	\end{proof}
	
	Consequently, $y_{\sigma} \in \pi_{\sigma}^{-1}(0)$ if and only if there exists a compact face $\Delta_J$ such that $y_{\sigma} \in E_{\sigma,J}$.
	
	For each $n$-dimensional cone $\sigma \in \Sigma$ and any generator $g_i$ of $I$, let $h_i(y_{\sigma})$ be the analytic germ of function given by
	$$ h_{i}(y_{\sigma})= \sum_{v\in supp(g_i)}a_v\cdot y_{\sigma,1}^{\langle a^1(\sigma),v\rangle-\ell(a^1(\sigma))}\dots y_{\sigma,n}^{\langle a^n(\sigma),v\rangle-\ell(a^n(\sigma))},$$
	where  $g_i= \sum_{k}a_k x^k$, $a_v=a_k$ and $v=k_1b_1+\cdots+k_rb_r$.  The following equality is satisfied by $h_i(y_{\sigma})$: 
	\begin{equation}\label{eq4}
		g_i  \circ \pi_{\sigma}(y_{\sigma})=y_{\sigma,1}^{\ell(a^{1}(\sigma))} \cdots y_{\sigma,n}^{\ell(a^{n}(\sigma))} \cdot h_i(y_{\sigma}).
	\end{equation}
	
	Furthermore, if $y_{\sigma} \in E_{\sigma,J}$ then 
	\begin{equation}\label{eq5}
		(g_i)_{\Delta_J}  \circ \pi_{\sigma}(y_{\sigma})=y_{\sigma,1}^{\ell(a^{1}(\sigma))} \cdots y_{\sigma,n}^{\ell(a^{n}(\sigma))} \cdot h_i(y_{\sigma}).
	\end{equation}
	
	Therefore, the next proposition gives us a necessary and sufficient condition so that a monomial $ x^m$ belongs to the integral closure of the ideal $I\subseteq \OX{S}$.
	\begin{proposition}
		Let $I$ be an ideal of $\OX{S}$ and $x^m \in \OX{S}$. Then $x^m \in \overline{I}$ if and only if
		$$\vert x^m\vert \circ \pi_{\sigma}(y_{\sigma}) \leq \mathcal{C} \cdot \displaystyle\sup_{i} \{\vert g_i\vert  \circ \pi_{\sigma}(y_\sigma)\},$$
		for all $n$-dimensional cone $\sigma \in \Sigma$ and all $y_{\sigma} \in \pi_{\sigma}^{-1}(U)$, where $U$  is a neighborhood of $0$ in $X(S)$.\end{proposition}
	\begin{proof}
		If $x^m \in \overline{I}$, then
		$$\vert x^m\vert  \leq \mathcal{C} \cdot \displaystyle\sup_{i} \{\vert g_i(x)\vert \},$$
		for all $x$ in a neighborhood $U$ of $0$.
		Thus, for all $y_{\sigma} \in \pi_{\sigma}^{-1}(U)$, we have
		$$\vert x^m  \circ \pi_\sigma ( y_{\sigma})\vert  \leq \mathcal{C}\cdot\displaystyle\sup_{i} \{\vert g_i  \circ \pi_\sigma(y_\sigma)\vert \}.$$
		Conversely, consider $\varphi: \C \to \mathcal{O}$ an analytic curve, where $\mathcal{O}$ is the dense open set in $X(S)$ obtained from the torus action (\ref{action}).
		If $t\in U_0$, where $U_0$ is a neighborhood of $0$ in $\C$, then $\varphi(t)\in \mathcal{O}$. Due to the definition of $\pi_\sigma$, there exists $y_\sigma \in \C^n(\sigma)$ such that $\varphi(t) = \pi_\sigma(y_\sigma)$.
		It follows that $y_\sigma \in  \pi_\sigma^{-1}(\mathcal{O})$.
		In this way, we have 
		$$ \vert x^m \circ \varphi(t)\vert   \leq \mathcal{C} \cdot \displaystyle\sup_ {i} \{\vert g_i \circ \varphi(t)\vert \}, $$
		for all $t \in U_0$.
		Therefore, $x^m \circ \varphi \in \overline{\varphi^{*}(I)}$, for every curve $\varphi$ in $\mathcal{O}$.
		We conclude, by Lemma $(3.3)$\footnote{We are applying Lemma (3.3) in the particular case where $M = I$ is an ideal.} in \cite{Gaffney1999}, that $x^m \in \overline{I}$.
	\end{proof}
	\begin{lemma}\label{lemma45}
		An ideal $I$ is non-degenerate if and only if $\sup_{i} \vert h_i(y_{\sigma})\vert \neq 0$ for all $n$-dimensional cone $\sigma$ of $\Sigma$ and all $y_\sigma \in \pi_{\sigma}^{-1}(0)$.
	\end{lemma}
	\begin{proof}
		Let $y_{\sigma} \in \pi^{-1}_{\sigma}(0)$. Then we choose $J$ as large as possible such that $y_{\sigma}\in E^{*}_ {\sigma,J}$. It follows from Lemma \ref{lemma42} that $\Delta_J= \bigcap_{i\in J} \Delta(a^i(\sigma))$ is a compact face of $\Poli(I)$. Since $I$ is non-degenerate, there is no common solution in $(\C^*)^n$ for the equations
		$$L\big(g_{1_{\Delta_J}}\big)= \cdots= L\big(g_{s_{\Delta_J}}\big)=0,$$
		where $g_1,\dots,g_s$ are generators of $I$.
		Consider $\tilde{\phi}: \C^{n} \to X(S)$ defined by $\tilde{\phi}(z)=(z^{b_1},\dots,z^{b_r}),$ where $b_1,\dots,b_r$ are generators of $S$. We have 
		$$ L(g_{i_{\Delta_J}}) \circ  \tilde{\phi}^{-1}\circ \pi_{\sigma}(y_{\sigma})=y_{\sigma,1}^{\ell(a^{1}(\sigma))} \cdots y_{\sigma,n}^{\ell(a^{n}(\sigma))} \cdot h_i(y_{\sigma}),$$
		for all $y_{\sigma} \in E_{\sigma,J}.$ Then, there is no common solution in $E^{*}_{\sigma, J}$ for the equations$${h_{1}}_{\vert_{E_{\sigma,J}}}=\cdots={h_{s}}_{\vert_{E_{\sigma,J}}}=0.$$
		Therefore, $\sup \vert h_i(y_{\sigma})\vert  \neq 0,\text{ for all } y_{\sigma} \in \pi^{-1}_{\sigma}(0).$
		
		Conversely, suppose that $I$ is degenerate. Then, there exist a compact face $\Delta$ of $\Poli(I)$ and a point $z^0$ in $(\C^{*})^n$ such that
		$$L(g_{1_{\Delta}})(z^0)=\cdots = L(g_{s_{\Delta}})(z^0)=0.$$
		Let $\sigma$ be an $n$-dimensional cone in $\Sigma$ and $a^1(\sigma),\dots, a^n(\sigma)$ corresponding primitive integers such that $$\Delta=\Delta_J = \bigcap_{i \in J} \Delta(a^i(\sigma)),$$
		for some $J \subset \{1,\dots,n\}$.
		Due to the definition of $\pi_\sigma$, there is $y^0 \in \C^n(\sigma)$ such that $\tilde{\phi}^{-1} \circ\pi_{\sigma}(y^0)=z^0$. Take $\overline{y}^0 \in E_{\sigma,J}$, where
		$$\overline{y_j}^0 =
		\begin{cases}
			y^0_j, & \text{if }j \notin J \\
			0, &  \text{if } j \in J.
		\end{cases}$$
		By Lemma \ref{lemma42} one has $\overline{y}^0 \in \pi^{-1}_{\sigma}(0)$, because $\Delta=\Delta_J$ is a compact face of $ \Poli(I)$. Therefore, $h_i(\overline{y}^0)=0$, for all $1\leq i\leq s$. We conclude that $\displaystyle\sup_{i}\{|h_i(\overline{y}^0)|\}=0$.
		
	\end{proof}

	\begin{remark}
		One can see $y_\sigma= 0\in E_{\sigma,J}$, where $J=\{1,\cdots,n\}$. Thus,
		$$h_i(0)= \sum_{v\in supp(g_i)\cap\Delta_J} a_v\cdot 0^{\langle a^1(\sigma),v\rangle-\ell(a^1(\sigma))}\cdots 0^{\langle a^n(\sigma),v\rangle-\ell(a^n(\sigma))}.$$
		
		Since $\Delta_J $ is a set of a single point, say $q$, then at least for some generator $g_i$, we have $h_i(0)=a_q \neq 0$. Therefore, $\sup_i\vert h_i(0) \vert \neq 0.$
	\end{remark}

	\begin{lemma}\label{lemma47}
		Let $I\subseteq \OX{S}$ be a non-degenerate ideal. Then, for each $m=(m_1,\dots,m_n) \in \Z^{n}_{+}$, $x^m \in \overline{I}$ if and only if the inequality 
		$$\ell(a^i(\sigma)) \leq \langle m_1b_1+\cdots+m_rb_r,a^i(\sigma)\rangle$$
		holds for all $1\leq i \leq n$ and all $n$-dimensional cone $\sigma\in \Sigma$.
	\end{lemma}
	\begin{proof}
		If $x^m \in \overline{I}$ then $m_1b_1+\cdots+m_rb_r \in \Poli(\overline{I})$. Since $\Poli(\overline{I})= \Poli(I)$, it follows that $\ell(v)\leq \langle v,m_1b_1+\cdots+m_rb_r\rangle$, for all $v \in \coned {S}$. In particular, $$\ell(a^i(\sigma))\leq \langle a^i(\sigma),m_1b_1+\cdots+m_rb_r\rangle,$$
		for all $1 \leq i \leq n$ and all cones $\sigma \in \Sigma$ of dimension $n$.
		For any $n$-dimensional cone $\sigma \in \Sigma$, we have
		$$\sup_{i} \{|g_i| \circ\pi_{\sigma}(y_\sigma)\}=|y_{\sigma,1}^{\ell(a^{1}(\sigma))} \cdots y_{\sigma,n}^{\ell(a^{n}(\sigma))}| \cdot \sup_{i}\{\vert h_i(y_{\sigma})\vert\}$$
		and 
		$$|x^m| \circ \pi_{\sigma}(y_{\sigma})= \vert y_{\sigma,1}^{\langle m_1b_1+\cdots+m_rb_r,a^{1}(\sigma)\rangle} \cdots y_{\sigma,n}^{\langle m_1b_1+\cdots+m_rb_r,a^{n}(\sigma)\rangle}\vert=\vert y_{\sigma,1}^{M_1}\cdots y_{\sigma,n}^{M_n}\vert, $$
		where $M_i:=\langle m_1b_1+\cdots+m_rb_r,a^{i}(\sigma)\rangle$.
		
		Since  $I$ is a non-degenerate ideal, by Lemma \ref{lemma45} one has $\displaystyle\sup_{i}\{\vert h_i(y_{\sigma})\vert\}> 0$, for all $y_{\sigma} \in \pi_{\sigma}^{-1}(0)$.  Then, there exists $\varepsilon >0$ such that
		$$ \vert y_{\sigma,1}^{M_1}\cdots y_{\sigma,n}^{M_n}\vert <\varepsilon \cdot \frac{ \sup_{i} \{|g_i| \circ\pi_{\sigma}(y_\sigma)\}}
		{\sup_{i}\{\vert h_i(y_{\sigma})\vert\}}. $$
		Hence, there exists $W \subset \C^{n}(\sigma)$ a neighborhood of $0$ such that $\displaystyle\sup_i\{\vert h_i(y_\sigma)\vert\}>0$, for all $y_\sigma \in W.$
		
		Take $K \subset W$ a compact set with $0 \in int(K)$. By continuity, there exists $\tilde{y} \in K $ such that $\sup_i\{\vert h_i(y_\sigma)\vert\} \geq \sup_i\{\vert h_i(\tilde{y})\vert\}, $ for all $y_{\sigma} \in K$.
		Define $d:=\sup_i\{\vert h_i(\tilde{y})\vert\} >0$ and $U_0:= int(K)$. We have that $U_0$ is a neighborhood of $0$ and $\sup_i\{\vert h_i(y_\sigma)\vert\} \geq d $, for all $y_{\sigma} \in U_0$.
		
		Now, we take $U:=\pi_\sigma(U_0)$. This implies $\sup_i\{\vert h_i(y_\sigma)\vert\} \geq d, $ for all $y_{\sigma} \in \pi_{\sigma}^{-1}(U)$. Thus,
		$$ \vert y_{\sigma,1}^{M_1}\cdots y_{\sigma,n}^{M_n}\vert \leq \frac{\varepsilon}{d} \cdot \sup_{i} \{ |g_i| \circ\pi_{\sigma}(y_\sigma)\}.$$
		Therefore, $x^m \in \overline{I}$, by Proposition \ref{prop21}.
	\end{proof}
	
	Now we present our main result, characterizing the ideals in  $\mathcal{O}_{X(S)}$  whose integral closure is generated by monomials with exponents that belong to the Newton polyhedron defined in  $\text{cone} (S) \subseteq \mathbb{R}^{n}_{+} $, where  $S$ is a finitely generated semigroup in $ \mathbb{Z}^{n}_{+}$.
	
	\begin{theorem}\label{theo48}
		Let $I$ be an ideal in $\mathcal{O}_{X(S)}$. Then $I$ is non-degenerate if and only if $\Gamma_{+}(I)=C(\overline{I}).$
	\end{theorem}
	\begin{proof}
		In Lemma $\ref{lemma34}$, we proved the inclusion $C(\overline{I}) \subset \Gamma_{+}(I)$. So, it is necessary to prove that $I$ is non-degenerate if and only if $\Poli(I)\subset C(\overline{I})$.
		Observe that $m_1b_1+\cdots+m_rb_r \in \Gamma_{+}(I)$ if, and only if, the inequality
		\begin{equation}\label{7}
			\ell(v^s) \leq\langle m_1b_1+\cdots+m_rb_r,v^s\rangle
		\end{equation}
		holds for all primitive integers $v^s$ such that $\Delta(v^s)$ is a $(n-1)$-dimensional face of $\Gamma_{+}(I)$.
		Thus, it remains to prove the inequality holds if and only if the inequality
		\begin{equation}\label{8}
			\ell(a^i(\sigma)) \leq \langle m_1b_1+\cdots+m_rb_r,a^i(\sigma)\rangle
		\end{equation}
		is satisfied for all $1 \leq i \leq n$ and for every cone $\sigma \in \Sigma$ of dimension $n$.
		
		By the construction of $\Sigma$, the set of primitive vectors $v^s$ such that $\Delta(v^s)$ is an $(n-1)$-dimensional face of $\Poli(I)$ is the set $\bigcup_{\sigma \in \Sigma} \left\{a^i(\sigma)\right\}$. Thus, the inequality (\ref{7}) is valid for all $v^s$ if the inequality (\ref{8}) is valid for all $a^i(\sigma)$.
		
		For any $n$-dimensional cone $\sigma \in \Sigma$, let $\sigma_0$ be the $n$-dimensional cone of $\Sigma_0$ such that $\sigma \subset \sigma_0$. Then, there exists a set of primitive integer vectors $\{v^1,\dots,v^p\}$ such that
		$\sigma_0$ is a linear combination of $v^s$ with coefficients in $\R_{+}$. If the vectors $v^s$ are numbered appropriately, we can assume that $\displaystyle\sigma_0= \sum^{p}_{s=1}\R_{+}v^s.$
		It follows that $\displaystyle\bigcap^{p}_{s=1} \Delta(v^s)=\bigcap^{n}_{i=1} \Delta(a^i (\sigma))$ is an one-point set in $\R^{n}_{+}$‚ say $q=(q_1,\dots,q_n)$. Thus, there are non-negative integers $\alpha_{is}$, $1\leq i \leq n$, such that $a^i(\sigma)= \sum^{p}_{s=1} \alpha_{is }v^s$.
		
		Since $q$ is a vertex of $\Gamma_{+}(I)$, we have $ \langle q,v^s \rangle = \ell(v^s)$ and $ \langle q,a^ i(\sigma) \rangle= \ell(a^i(\sigma))$,
		for $1\leq s\leq p $ and $1\leq i\leq n$. Hence, $\ell(a^i(\sigma))= \sum^{p}_{s=1} \alpha_{is}l(v^s).$
		If the inequality $\ell(v^s)\leq \langle m_1b_1+\cdots+m_rb_r, v^s\rangle$ holds, then
		$$ \ell(a^i(\sigma))= \sum^{p}_{s=1} \alpha_{is}l(v^s)\leq\sum^{p}_{s=1} \alpha_{is}\langle m_1b_1+\cdots+m_rb_r,v^s \rangle =\langle m_1b_1+\cdots+m_rb_r,a^i(\sigma) \rangle.$$

		\noindent Hence, by Lemma \ref{lemma47}, $x^m \in \overline{I}$. Therefore, $m_1b_1+\cdots+m_rb_r \in C(\overline{I})$.
		
		On the other hand, consider $\sigma$ an $n$-dimensional cone of $\Sigma$. Then, there exists a vertex $Q=q_1b_1+\cdots+q_rb_r$ of $\Poli(I)$ such that $Q \in \bigcap^{n}_{i=1} \Delta(a^i(\sigma) )$. Since $C(\overline{I}) = \Gamma_{+}(I)$, each vertex $Q$ of $ \Gamma_{+}(I)$ is a vertex of $C(\overline{I}) $. By the definition of $C(\overline{I})$, it follows that $x^q \in \overline{I}$, where $q=(q_1,\dots,q_n)$.
		
		Therefore, there exists a constant $\mathcal{C}>0$ such that $$|x^q| \leq \mathcal{C}\cdot \sup_{i} \{ |g_i(x)|\}$$ in a neighborhood $U$ of the origin in $X(S)$, where $g_1,g_2, \dots, g_s$ are generators of $I$.
		
		Hence, for all $y_{\sigma}\in \pi_{\sigma}^{-1}(U)$, we have
		\begin{eqnarray*}
			|x^q|  \circ \pi_{\sigma}(y_{\sigma})= |y_{\sigma,1}^{M_1} \cdots y_{\sigma,n}^{M_n}| &\leq& \mathcal{C} \cdot\sup_{i}\{ |g_i| \circ \pi_{\sigma}(y_{\sigma})\}\\&=& \mathcal{C}\cdot |y_{\sigma,1}^{\ell(a^1(\sigma))} \cdots y_{\sigma,n}^{\ell(a^n(\sigma))}| \cdot \sup_{i} \{|h_i(y_{\sigma})|\}, 
		\end{eqnarray*}
		where $M_i= \langle Q,a^{i}(\sigma)\rangle=\ell(a^i(\sigma))$.
		
		Then, $\sup_{i} \{|h_i(y_{\sigma})|\}> 0$ in a neighborhood of $\pi_{\sigma}^{-1}(0)$. Therefore, by Lemma $\ref{lemma45}$, the ideal $I$ is non-degenerate.
	\end{proof}
	
	\begin{corollary}\label{corollary48}
		Let $I$ be an ideal in $\mathcal{O}_{X(S)}$. Then $I$ is non-degenerate if and only if $\overline{I}=I^{\circ}.$
	\end{corollary}
	\begin{proof}
		If $\Poli(I)=C(\overline{I})$ then $I^{\circ}=\langle x^k :k_1 b_1+\cdots+ k_r b_r \in C(\overline{I})\rangle=\langle x^k : x^k \in \overline{I}\rangle$. However, $\overline{I}=I^{\circ}.$ Conversely, if $\overline{I}=I^{\circ}$ then $K_I=\langle x^m : x^m \in I^{\circ} \rangle = I^{\circ}$. Thus, we have 
		$C(\overline{I})=\Poli(k_I)=\Poli(I^{\circ})=\Poli(I).$ Therefore, the result follows from Theorem \ref{theo48}.
	\end{proof}
	
	\begin{example}
		Let $S=\langle (1,0),(1,1),(1,2) \rangle$. Consider  $I \subset \mathcal{O}_{X(S)}$ the ideal generated by $f=x^4y+z^2 $ and $g=xyz+x^2y^3$. Then, $supp(I)=\{(5,1),(2,4),(3,3),(5,3)\}.$ Thus, the Newton polyhedron of $I$ is represented in the Figure \ref{Fig5}
		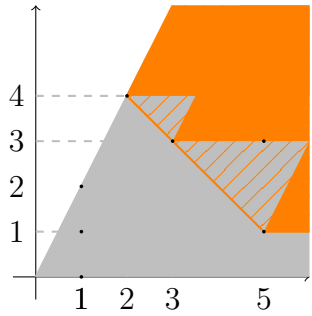
\begin{figure}[H]
			\centering
			
			\begin{tikzpicture}[scale=0.6]
				\draw[->] (-0.5,0) -- (6,0);
				\draw[->] (0,-0.5) -- (0,6);

				\draw[lightgray, thick] (0,0) -- (2.5,5);
				
				\fill[lightgray, opacity=0.2](0,0) -- (3,6) -- (6,6) --(6,0)  --cycle;
				\fill[black] (1,2) circle (0.4mm) ;
				\fill[black] (1,1) circle (0.4mm) ;

				\fill[black] (0,2) node[left] {$2$};
				\fill[black] (0,4) node[left] {$4$};
				\fill[black] (0,3) node[left] {$3$};
				
				\fill[black] (0,1) node[left] {$1$};
				\fill[black] (1,0) circle (0.4mm) node[below] {$1$};
				\fill[black] (1,0)  node[below] {$1$};
				\fill[black] (2,0)  node[below] {$2$};
				\fill[black] (3,0)  node[below] {$3$};
				\fill[black] (5,0)  node[below] {$5$};

				\draw[orange, thick] (2,4) -- (3,6);
				\draw[orange, thick] (2,4) -- (5,1);
				\draw[orange, thick] (5,1) -- (6,1);
				\fill[orange, opacity=0.2] (2,4) -- (3,6) -- (6,6) -- (6,4)--cycle;
				\fill[orange, opacity=0.2] (3,3) -- (4.5,6) -- (6,6) -- (6,3)--cycle;
				\fill[orange, opacity=0.2] (5,1) -- (6,3) -- (6,1) --cycle;
				\pattern[pattern={hatch[hatch size=7pt, hatch linewidth=.2pt, hatch angle=90]} , pattern color=orange] (2,4) -- (3,6) --(6,6)--(6,1)--(5,1) --cycle;

				\draw[lightgray,  thick, dashed] (0,4) -- (2,4); 
				\draw[lightgray,  thick, dashed] (2,0) -- (2,4);
				\fill[black] (2,4) circle (0.4mm);
				
				\draw[lightgray,  thick, dashed] (3,0) -- (3,3);
				\draw[lightgray,  thick, dashed] (0,3) -- (3,3);
				\fill[black] (3,3) circle (0.4mm) ;
				
				\draw[lightgray,  thick, dashed] (5,0) -- (5,1);
				\draw[lightgray,  thick, dashed] (0,1) -- (5,1);
				\fill[black] (5,1) circle (0.4mm);
				
				\fill[black] (5,3) circle (0.4mm);
				
			\end{tikzpicture}
			\caption{Newton polyhedron of  $I$.}
			\label{Fig5}
		\end{figure}
		
		We observe that $I$ is non-degenerate. Then, we have that the integral closure of $I$ is generated by the monomials $x^{k}$ such that
		$k_{1}b_{1}+\cdots+k_{r}b_{r} \in \Poli(I)$, by Theorem \ref{theo48}.
	\end{example}

Following this, we present an example addressing non-degenerate ideals in relation to Whitney equisingularity. Before that, we will recall the definition pertaining  to the distance between linear subspaces.

Suppose $A$, $B$ are linear subspaces at the origin in $\C^{n}$, then 
$$dist(A,B) = \sup_{ u\in B^{\perp}-\{0\}, v\in A-\{0\}}
\frac{|(u,v)|}{\left\| u \right\|
	\left\| v \right\|},$$
	where $B^{\perp}$ denotes the orthogonal space of $B$.

\begin{definition} Let $X$, $Y$ be strata in a stratification of a complex analytic space, satisfying $\overline{X}\supset Y$. Then the pair $(X,Y)$ satisfies the Verdier's condition $W$ at $0 \in Y$ if there exists a positive real number $C$ such that
	$$dist(T_{0}Y, T_{x}X) \leq \mathcal{C}dist(x,Y),$$
	for all $x$ close to $Y$.
\end{definition}

In the complex analytic context, Verdier in \cite{Verdier1976} and Teissier in \cite{Teissier1982} proved the equivalence between the Verdier's condition $W$ at $0$ and Whitney's conditions (a) and (b) at $0$.

Teissier shows the importance of the integral dependence relation in the determinacy of the Whitney conditions in the case  of a family of analytic hypersufarces \cite{Teissier1982}, and Gaffney generalized Teissier's result for any codimension \cite{Gaffney1992} in the following result.

\begin{theorem}\textup{\cite{Gaffney1992}}\label{Theorem51}
	Let $X$ be a complex analytic, reduced, purely $d$ dimensional space, $Y$ an analytic space of $X$ purely of dimension $t$, $0$ a smooth point of $Y$, and $X_0$ the set of smooth points on $X$. Let $F:\C^t \times \C^{r} \to \C^{p}$ be coordinates chosen so that $\C^{t} \times \{0\} = Y$, $F$ defines $X$ with reduced structure. Then $\frac{\partial F}{\partial y_l} \in \overline{\big\{x_i \frac{\partial F}{\partial x_j}\big\}\mathcal{O}_{X}^{p}}$ for all $l=1,\dots,t$, in which $i,j=1, \dots, r$ if, and only if, $(X_0,Y)$ satisfies Verdier's condition $W$.
\end{theorem}

Using the notation of Gaffney's result and Theorem \ref{theo48} we present the next example.

 \begin{example}\label{ExW}
	Let $F:\C \times \C^{2} \to \C$ be defined by $F(y_{1},x_{1},x_{2})=x_{1}^{4}y_{1}^{5}-x_{2}^{3}$. Note that $F^{-1}(0)=X(S)$, where $S=\langle (1,3), (1,0),(3,5)\rangle$. Consider the ideal $I_F$ generated by $\big\{x_i \frac{\partial F}{\partial x_j}\big\}$, i.e., 
	$$I_F=\langle 4x_{1}^{4}y_{1}^{5},4x_{1}^{3}x_{2}y_{1}^{5},-3x_{1}x_{2}^{2},-3x_{2}^{3}\rangle.$$ Thus, we have $supp(I_F)=\langle (9,15),(11,20),(7,10)\rangle$. Then, the Newton polyhedron of $I_F$ is represented in the Figure \ref{Figex}.
	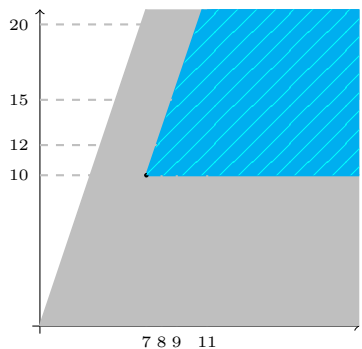
\begin{figure}[H]
		\centering
		
		\begin{tikzpicture}[scale=0.2]
			\draw[->] (-0.5,0) -- (21,0);
			\draw[->] (0,-0.5) -- (0,21);

			\draw[lightgray, thick] (0,0) -- (7,21);
			
			\fill[lightgray, opacity=0.2](0,0) -- (7,21) -- (21,21) --(21,0)  --cycle;

			\fill[black] (0,20) node[left] {$\scriptscriptstyle 20$};
			\fill[black] (0,15) node[left] {$\scriptscriptstyle 15$};
			\fill[black] (0,12) node[left] {$\scriptscriptstyle 12$};
			\fill[black] (0,10) node[left] {$\scriptscriptstyle 10$};

			\fill[black] (9,0)  node[below] {$\scriptscriptstyle 9$};
			\fill[black] (7,0)  node[below] {$\scriptscriptstyle 7$};
			\fill[black] (8,0)  node[below] {$\scriptscriptstyle 8$};
			\fill[black] (11,0)  node[below] {$\scriptscriptstyle 11$};

			\draw[cyan, thick] (7,10) -- (10.7,21);
			\draw[cyan, thick] (7,10) -- (21,10);
			
			\draw[lightgray,  thick, dashed] (0,15) -- (9,15); 
			\draw[lightgray,  thick, dashed] (9,0) -- (9,15);
			\fill[black] (9,15) circle (1.5mm) ;
			\draw[lightgray,  thick, dashed] (0,20) -- (11,20); 
			\draw[lightgray,  thick, dashed] (11,0) -- (11,20);
			\fill[black] (11,20) circle (1.5mm) ;
			\draw[lightgray,  thick, dashed] (0,10) -- (7,10) ; 
			\draw[lightgray,  thick, dashed] (7,0) -- (7,10) ;
			\fill[black] (7,10) circle (1.5mm) ;
			\draw[lightgray,  thick, dashed] (0,12) -- (8,12) ; 
			\draw[lightgray,  thick, dashed] (8,0) -- (8,12) ;
			\fill[red] (8,12) circle (2mm) ;

			\fill[cyan, opacity=0.2] (9,15) -- (11,21) -- (21,21) -- (21,15)--cycle;
			\fill[cyan, opacity=0.2] (7,10) -- (10.7,21) -- (21,21) -- (21,10)--cycle;
			\fill[cyan, opacity=0.2] (11,20) -- (11.3,21) -- (21,21) -- (21,20)--cycle;
			\pattern[pattern={hatch[hatch size=7pt, hatch linewidth=.2pt, hatch angle=90]} , pattern color=cyan]  (7,10) -- (10.7,21) -- (21,21) -- (21,10)--cycle;

		\end{tikzpicture}
		\caption{Newton polyhedron of  $I_F$.}
		\label{Figex}
	\end{figure}
	Consider the compact face $\Delta=\{(7,10)\}$ and the polynomials \begin{eqnarray*}
		L(4x_{1}^{4}y_{1}^{5})_{\Delta}= 0 &\ \ &  L(4x_{1}^{3}x_{2}y_{1}^{5})_{\Delta}= 0\\
		L(-3x_{1}x_{2}^{2})_{\Delta}= z_{1}^{7}z_{2}^{10} &\ \ & L(-3x_{2}^{3})_{\Delta}= 0
	\end{eqnarray*}
	
	Note that $$L(4x_{1}^{4}y_{1}^{5})_{\Delta}=  L(4x_{1}^{3}x_{2}y_{1}^{5})_{\Delta}= L(-3x_{1}x_{2}^{2})_{\Delta}= L(-3x_{2}^{3})_{\Delta}=0$$
	have no common solution in $(\C^{\ast})^{2}$. Then, $I_F$ is non-degenerate.
	Thus, since $(8,12) \in \Gamma_{+}(I_F)$, it follows from the Theorem \ref{theo48} that
	$$ \frac{\partial F}{\partial y_{1}}=5x_{1}^{4}y_{1}^{4} \in \overline{I_F}.$$
	Therefore, the smooth part of $X(S)$ and $\C\times 0$ satisfies Verdier's condition $W$, by Theorem \ref{Theorem51}.
\end{example}

	\section*{Acknowledgements}
	
	The authors are very grateful to Carles Bivià-Ausina, who in one of his visits to São Carlos gave us the idea of studying this kind of problem using the combinatory from the toric varieties and semigroups.
	
	Amanda S. Araújo is supported by CAPES grant number 88887.827300/2023-00. Tha\'is M. Dalbelo is supported by FAPESP-Grants 2019/21181-0 and 2024/22060-0 and by CNPq grant 403959/2023-3. Thiago da Silva is funded by CAPES grant number 88887.909401/2023-00 and CAPES grant number 88887.897201/2023-00.

	\bibliographystyle{plain}
	\bibliography{references}
	
	\vspace{2cm}

	\textsc{Amanda S. Araújo (Federal University of São Carlos)}
	
	 amandaaraujo@estudante.ufscar.br
	 
	 \vspace{0.5cm}
	
	\textsc{Thaís M. Dalbelo (Federal University of São Carlos)}
	
	thaisdalbelo@ufscar.br
	
	\vspace{0.5cm}
	
	\textsc{Thiago da Silva (Federal University of Espírito Santo)}
	
	thiago.silva@ufes.br

\end{document}